\documentclass{amsart}
\usepackage[cp1251]{inputenc}
\usepackage[english,russian]{babel}
\usepackage{amsmath}
\usepackage{amssymb}
\usepackage{amsfonts}

\newtheorem{lemma}{Лемма}
\newtheorem{theorem}{Теорема}
\newtheorem{definition}{Определение}
\newtheorem{corollary}{Следствие}
\newtheorem{proposition}{Предложение}

\begin{document}
\thispagestyle{empty}

\title[Estimates for generalizated oscillatory integrals ...]
{Estimates for generalizated oscillatory integrals with polynomial phase}

\author{ I.A.Ikromov, A.R.Safarov}
\address{Isroil A.Ikromov, Akbar R.Safarov, 
\newline\hphantom{iii}Uzbekistan Academy of Sciences V.I.Romanovskiy Institute of Mathematics,
\newline\hphantom{iii}Samarkand State University 
\newline\hphantom{iii} 15 University Boulevard 
\newline\hphantom{iii} Samarkand, 140104, Uzbekistan}
\email{safarov-akbar@mail.ru}

\thanks{\sc Ikromov I.A., Safarov A.R. 
Estimates for generalizated oscillatory integrals with polynomial phase}


\maketitle
{
\small
\begin{quote}
\noindent{\bf Abstract. } In this paper we consider the problem on uniform estimates for generalized oscillatory integrals given by   Mittag- Leffler functions with the homogeneous polynomial phase. We obtain a variant of Ricci-Stein Lemma and invariant estimates for corresponding integrals.
\medskip

 \noindent{\bf Keywords:}  Mittag-Leffler functions, phase  function, amplitude.
\end{quote}
}

\section{Introduction}\label{www}
Many problems of harmonic analysis, analytic number theory, and mathematical physics involve
trigonometric (oscillatory) integrals with polynomial phase, a common problem
integration of rational polynomials instead of integrating functions
  \cite{Safarov Summir}, \cite{Safarov},\cite{Safarov11},\cite{Safarov1},\cite{Ricci
and Stein}, \cite{Xitoy},\cite{Domar}. In harmonic analysis, estimates for one dimensional oscillatory integrals can be obtained using van der Corput lemma \cite{VanDer}.  In \cite{MichaelRuzhansky2012} a multidimensional version of the van der
Corput lemma is considered where the decay of the oscillatory integral is established with respect to
all space variables, combining the standard one-dimensional van der Corput lemma
with the stationary phase method.  Estimates for oscillatory integrals with polynomial phase can be found, for instance, in \cite{Safarov1},\cite{AKC}.

The function $E_{\alpha}(z)$ is named after the great Swedish mathematican G\"{o}sta Magnus Mittag-Leffler (1846-1927) who defined it by a power series
\begin{equation}\label{Formul3}
E_{\alpha}(z)=\sum_{k=0}^{\infty}\frac{z^{k}}{\Gamma(\alpha k+1)},\,\,\ \alpha\in\mathbb{C}, Re(\alpha)>0,
\end{equation}
and studied its properties in 1902-1905 in five subsequent notes \cite{ML1}-\cite{ML4} in connection with his summation method for divergent series.

 A classic generalizations of the Mittag-Leffler function, namely the two-parametric Mittag-Leffler function
\begin{equation}\label{Formul4}
E_{\alpha,\beta}(z)=\sum_{k=0}^{\infty}\frac{z^{k}}{\Gamma(\alpha k+\beta)},\,\,\ \alpha,\beta\in\mathbb{C}, Re(\alpha)>0,
\end{equation}
which was deeply investigated independently by Humbert and Agarval in \cite{Hum53},\cite{HumAga53},\cite{Aga53} and by Dzherbashyan in
\cite{Dzh54a},\cite{Dzh54b},\cite{Dzh54c},\cite{Rudolf}.

In the current paper we replace exponential function with the Mittag-Leffler-type function and study the "generalized" oscillatory integrals. In \cite{Ruzhansky} and \cite{Ruzhansky2021}  analogues of the van der Corpute lemmas involving Mittag-Leffler functions for one dimensional integrals have been considered. We consider estimates for multidimensional generalization oscillatory  integrals with polynomial phase. This work is analogous to
\cite{Ricci and Stein} and application for oscillatory integrals with Mittag-Leffler functions.

\section{Preliminaries}

\begin{proposition}(\cite{Pod})\label{Prop1}.
If $0<\alpha<2,\beta$ is an arbitrary real number and $\mu$ is such that $\pi\alpha/2<\mu<\min\{\pi,\pi\alpha\},$ then there is $C>0$ such that
\begin{equation}\label{Formul1}
  |E_{\alpha,\beta}(z)|\leq\frac{C}{1+|z|}, z\in\mathbb{C},\mu\leq|\arg(z)|\leq\pi.
\end{equation}
\end{proposition}

Let $P(a,x)=\sum\limits_{|\lambda|\leq d}a_{\lambda}x^{\lambda}$  denote a polynomial in $\mathbb{R}^{n}$ of degree at most $d,$ where we write $x^{\lambda}=x_{1}^{\lambda_{1}}\dots x_{n}^{\lambda_{n}},$ $\lambda=(\lambda_{1},\dots,\lambda_{n}),$ with $|\lambda|=\lambda_{1}+\lambda_{2}+\dots\lambda_{n}$.

\begin{definition}
A generalization of oscillatory integral with phase $P(a,x)$ and amplitude $\psi(x)$ is an integral of the form
\begin{equation}\label{int16}
I_{\alpha,\beta}(a)=\int _{Q^{n} }E_{\alpha,\beta}(iP(a,x))\psi(x)dx,
\end{equation}
where $0<\alpha<1,$ $\beta>0$, $\psi\in C^{\infty}(\mathbb{R}^{n})$, $Q^{n}:=[0,1]^n$ is $n$ dimensional cube and $P(a,x)$ polynomial. In particular if $\alpha=1$ and $\beta=1$ we have a classical oscillatory integral.
\end{definition}

\begin{theorem}(\cite{James Wright})\label{th3366}
For each $d,n$ there exist a finite constant $C:=C(n,d)$  such that for any multi-index $\kappa$ and any polynomial $P:\mathbb{R}^{n}\rightarrow\mathbb{R}$ of degree $\leq d$ satisfying $|D^{\kappa}P(a,x)|\geq1$ for every $x\in Q^{n}:=[0,1]^{n},$ and for any $\mu>0,$
$$|\{x\in Q^{n}:|P(a,x)|\leq\mu\}|\leq C\mu^{1/|\kappa|},$$
where $\kappa$ depends on $a.$
\end{theorem}

\begin{corollary}(\cite{James Wright})\label{cor3366}
Let $P(a,x)=\sum\limits_{0<|\lambda|\leq d}a_{\lambda}x^{\lambda}$ is the polynomial.  Then
$$\left|\int_{Q^{n}}e^{iP(a,x)}\right|dx\leq C_{d,n}\left(\sum\limits_{0<|\lambda|\leq d}|a_{\lambda}|\right)^{-\frac{1}{d}}.$$
Moreover $C_{d,1}\leq Cd$ for an absolute constant $C$.
\end{corollary}

\begin{lemma}\label{A33}
Let  $P:\mathbb{R}^{n}\rightarrow\mathbb{R}$ be a polynomial of degree $\leq d$  and $|a|=\max\{|a_{\lambda}|,\lambda\leq d\}$. There exists a constant $C_{d,n}$ such that if $a^{0}\in S^{N}$ (where $N+2$ is the dimension of space of polynomials of degree at most $d$) be a fixed point and $|D^{\kappa}P(a^{0},x)|\geq\delta>0$ then the following inequality holds
$$|I_{\alpha,\beta}(\mu a^{0})|\leq \frac{C_{d,n}}{|\mu|^{1/|\kappa|}},$$
where $S^{N}=\{|a|=1\}$ is the unit  sphere  with respect to metric $l^{1}$ and $\mu>0.$
\end{lemma}
\begin{proof}
Using Proposition 1 we obtain

$$|I_{\alpha,\beta}(\mu a^{0})|=\left|\int _{Q^{n} }E_{\alpha,\beta}(iP(a^{0}\mu,x))\psi(x)dx\right|\leq
\int_{Q^{n}}\frac{|\psi(x)|dx}{1+|P(a^{0}\mu,x)|}:=J(\mu a^{0}).$$

Now we represent $Q^{n}$ as union of sets $$Q^{n}=\Omega\cup A_{k},$$
where
$$\Omega=\{x\in Q^{n}:|P(a^{0}\mu,x)|\leq2\}=\left\{x\in Q^{n}:\left|P\left(\frac{a^{0}}{\mu},x\right)\right|\leq\frac{2}
{\mu}\right\}$$ and
$$A_{k}=\{x\in Q^{n}:2^{k}\leq |P(a^{0}\mu,x)|\leq2^{k+1}\}=\left\{x\in Q^{n}:\frac{2^{k}}{\mu}\leq\left|P\left(\frac{a^{0}}{\mu},x\right)\right|\leq\frac{2^{k+1}}
{\mu}\right\},$$ where $k=1,2,....$

 First we represent the integral $J(\mu a^{0})$ as  $$J(\mu a^{0})=J_{0}+J_{k}:=\int_{\Omega}+\int_{A_{k}}.$$ Now we estimate the integral $J_{0}$ on
the set $\Omega.$

So, we obtain
$$|J_{0}|\leq\int_{\Omega}\frac{|\psi(x)|dx}{1+|P(a^{0}\mu,x)|}.$$

Due to Corollary \ref{cor3366} we have
$$|J_{0}|\leq C|\Omega|\leq\frac{C}{\mu^{\frac{1}{|\kappa|}}}.$$

Then we consider the integral $I_{\alpha,\beta}$ on
the sets $A_{k}$. By Theorem \ref{th3366} and we have: $$\left|\left\{\left|P\left(\frac{a^{0}}{\mu},x\right)\right|\leq\frac{2^{k+1}}{\mu},
x\in Q^{n}\right\}\right|
\leq C\left(\frac{2^{k+1}}{\mu}\right)^{\frac{1}{|\kappa|}}.$$

We write $J(\mu a^{0})$ as sum of integrals $J_{k}$ and obtain
$$J(a^0\mu)-J_0(a^0\mu)=\sum\limits_{2^{k}\leq|P(a^{0}\mu,x)|\leq2^{k+1}} J_{k}=\sum\limits_{2^{k}\leq|P(a^{0}\mu,x)|\leq2^{k+1}}\int_{A_{k}}\frac{|\psi(x)|}{1+|P(a^{0}\mu,x)|}dx$$ using  Theorem \ref{th3366} we find the following estimate:
\begin {eqnarray*}|J_{k}|=\left|\int_{A_{k}}\frac{\psi(x)}{1+|P(a^{0}\mu,x)|}dx\right|\leq\left(\frac{2^{k+1}}{\mu}\right)^{\frac{1}{|\kappa|}}
\frac{C}{2^{k}}.
\end {eqnarray*}
From here we find the sum $J_{k}$ and, by estimating the integral $I_{\alpha,\beta}(a^{0}\mu).$
So
$$|I_{\alpha,\beta}(a^{0}\mu)|=|J_{0}|+\sum_{k=1}^{\infty}J_{k}\leq\frac{C}{\mu^{\frac{1}{|d|}}}+\sum_{k=1}^{\infty}\left(\frac{2^{k+1}}{\mu}\right)^{\frac{1}{|\kappa|}}
\frac{C}{2^{k}}\leq\frac{C}{\mu^{\frac{1}{|\kappa|}}}+\frac{C}{\mu^{\frac{1}{|\kappa|}}}\sum_{k=1}^{\infty}2^{\frac{k+1}{|\kappa|}-k}.$$
As the last series is convergence since $|\kappa|\geq2$ then we obtain proof of Lemma \ref{A33}.
\end{proof}

\section{Relation to Mittag-Leffler functions}

\begin{theorem} \label{Th99}
Let $0<\alpha<1,$ $\beta>0$. There exists positive number $C_{d}$ such that for the integral \eqref{int16} with phase  $P(a,x)$ following inequality holds
\begin{equation}\label{Formul199}
|I_{\alpha,\beta}|\leq\frac{C_{d}}{\|a\|^{\frac{1}{d}}},
\end{equation}
where $\|a\|=\sum\limits_{|\lambda|\leq d}|a_{\lambda}|.$
\end{theorem} \label{th2}

\textbf{Remark.}
Theorem \ref{Th99} is an analog of the Ricci-Stein lemma \cite{Ricci and Stein}.

\textbf{Proof of Theorem \ref{Th99}.}
We use inequality \eqref{Formul1} for the integral \eqref{int16} and obtain
\begin{align}
|I_{\alpha,\beta}|=& \left|\int _{Q^{n} }E_{\alpha,\beta}(i P(a,x))\psi(x)dx\right|\leq\int _{Q^{n} }\left|E_{\alpha,\beta}(iP(a,x) )\right|\left|\psi(x)\right|dx\nonumber
 \\
\leq & C  \int_{Q^{n}} \frac{|\psi(x)|dx}{1+|P(a,x)|}. \label{int_estimaqqq}
\end{align}

We represent cube as union of following subspaces $$Q^{n}=\Delta_{1}\cup\Delta_{2}:=\{x\in Q^{n}:|P(a,x)|<2\}\cup\{x\in Q^{n}:|P(a,x)|\geq2\}.$$
First we estimate the last integral \eqref{int_estimaqqq} over the $\Delta_{1}$
$$|I_{\alpha,\beta}|\leq\int_{\Delta_{1}}\frac{|\psi(x)|dx}{1+|P(a,x)|}.$$
Using the results sublevel set estimates Theorem \ref{th3366} and we have
$$\left|\int_{\Delta_{1}}\frac{|\psi(x)|dx}{1+|P(a,x)|}\right|\leq\frac{c\max\limits_{x\in Q^{n}}|\psi(x)|}{|a|^{\frac{1}{|\kappa|}}}\leq\frac{c}{|a|^{\frac{1}{|d|}}}.$$
Now we estimate the integral \eqref{int_estimaqqq} over the $\Delta_{2}$
$$|I_{\alpha,\beta}|\leq\int_{\Delta_{2}}\frac{|\psi(x)|dx}{1+|P(a,x)|}.$$
We use Lemma \ref{A33} and we get
$$|I_{\alpha,\beta}|\leq\frac{c}{|a|^{\frac{1}{|d|}}}.$$

\section{The case $d=3$}

Now we consider the following example for $d=3$.

Let $0<\delta<1.$ We consider the integral:
\begin{equation}\label{int188}
J_{\alpha,\beta}=\int_{|x|\leq1} E_{\alpha,\beta}(i(x^{3}+px+q))dx,
\end{equation}
where $0<\alpha<1,$ $\beta>0$.

\begin{theorem}\label{th203}
\noindent \textit{1) If $\frac{1}{3} \leq\delta \leq\frac{1}{2} $ then for the integral \eqref{int188} following estimate holds}
\begin{equation} \label{GrindEQ__3_3_}
\left|J_{\alpha,\beta}\right|\le \frac{c_{\delta } }{\left(\frac{|p|^{3} }{27} +\frac{q^{2} }{4} \right)^{\frac{3\delta -1}{6} } } ,
\end{equation}
\textit{2) If $\frac{1}{2} <\delta <1$ then for the integral \eqref{int188} following estimate holds }
\begin{equation} \label{GrindEQ__3_4_}
\left|J_{\alpha,\beta}\right|\le \frac{c_{\delta } }{|D|^{\delta -\frac{1}{2} } \left(\frac{|p|^{3} }{27} +\frac{q^{2} }{4} \right)^{\frac{2-3\delta }{6} } } ,
\end{equation}
where $D=\frac{p^{3}}{27}+\frac{q^{2}}{4}$ is discriminant of polynomial
$x^{3}+px+q$ and $c_{\delta}$ any positive only depending to $\delta.$
\end{theorem}

\textbf{Remark.} If $\delta\leq\frac{1}{3},$ then according to Theorem
\ref{Th99} the integral \eqref{int188} bounded.

\textbf{Proof of Theorem \ref{th203}.} If $p=q=0,$ then the required estimate is trivially satisfied.
We use inequality \eqref{Formul1} for the integral \eqref{int188} and we obtain
\begin{equation}\label{integ2}
|J_{\alpha,\beta}|\leq\int\limits_{|x|\leq1}\frac{dx}{1+|x^{3}+px+q|}
\leq\int\limits_{|x|\leq1}\frac{dx}{|x^{3}+px+q|^{\delta}}:=J(p,q).
\end{equation}

Let $\frac{1}{3} <\delta <\frac{1}{2} $ and $(p,q)\ne (0,0).$ Now we consider the following cases separately.
\begin{equation} \label{GrindEQ__3_5_}
\max \left\{\frac{|p|^{3} }{27} ,\frac{q^{2} }{4} \right\}=\frac{|p|^{3} }{27}
\end{equation}
and
\begin{equation} \label{GrindEQ__3_6_}
\max \left\{\frac{|p|^{3} }{27} ,\frac{q^{2} }{4} \right\}=\frac{q^{2} }{4} .
\end{equation}

Suppose the condition \eqref{GrindEQ__3_5_} holds, then we use change the variables  $x=\left|p\right|^{\frac{1}{2} } y$ and obtain:
\[\left|J(p,q)\right|=\int _{-\left|p\right|^{-\frac{1}{2} } }^{\left|p\right|^{-\frac{1}{2} } } \frac{\left|p\right|^{\frac{1}{2} } dy}{\left|\left|p\right|^{\frac{3}{2} } y^{3} +\left|p\right|^{\frac{1}{2} } py+q\right|^{\delta } } =\frac{1}{\left|p\right|^{\frac{3\delta -1}{2} } } \int _{-\left|p\right|^{-\frac{1}{2} } }^{\left|p\right|^{-\frac{1}{2} } } \frac{dy}{\left|y^{3} +sgn(p)y+q\left|p\right|^{-\frac{3}{2} } \right|^{\delta } } .\]

We consider the integral
\begin{equation} \label{GrindEQ__3_7_}
\left|J_{1} (p,q)\right|=\int _{-\left|p\right|^{-\frac{1}{2} } }^{\left|p\right|^{-\frac{1}{2} } } \frac{dy}{\left|y^{3} +sgn(p)y+B\right|^{\delta } } ,
\end{equation}
where $\left|B\right|=\left|\frac{q}{\left|p\right|^{\frac{3}{2} } } \right|\le \frac{2}{3\sqrt{3} } .$ We show that this integral is bounded, when $\frac{1}{3}<\delta <\frac{1}{2} .$ If $\left|y\right|\ge 2,$ then $\left|y\right|^{3} \left|1+\frac{1}{y^{2} } +\frac{B}{y^{3} } \right|\ge \left|y\right|^{3} \left|1-\frac{1}{4} -\frac{1}{12\sqrt{3} } \right|\ge \frac{\left|y\right|^{3} }{2} .$
Consequently,
\begin{equation} \label{GrindEQ__3_8_}
\left|J(p,q)\right|=\int _{-2}^{2} \frac{dy}{\left|y^{3} +sgn(p)y+B\right|^{\delta } } +R(B),
\end{equation}
where $R(B)$ is a bounded function of $B\in \left[-\frac{2}{3\sqrt{3} } ,\frac{2}{3\sqrt{3} } \right].$

Let
\begin{equation} \label{GrindEQ__3_9_}
\left|J_{10} \right|:=\int _{-2}^{2} \frac{dy}{\left|y^{3} +sgn(p)y+B\right|^{\delta } } .
\end{equation}
     Suppose $sgn(p)=1.$ Then polynomial $y^{3} +y+B$ has no multiple roots. In fact, for any $B$ it has one simple root. Thus,
\[\left|J_{0} \right|=\int _{-2}^{2} \frac{dy}{\left|y^{3} +y+B\right|^{\delta } } \le c,\]
for $\delta <\frac{1}{2} <1.$

Suppose $sgn(p)=-1.$ Then, since $B$ below in a compact set, it suffices to obtain the corresponding "local" estimate.  Let $\varepsilon $ is sufficiently small positive number.  For a number $B$, consider two sets
\[A_{1} =\left\{\left|B^{2} -\frac{4}{27} \right|\ge \varepsilon \right\},A_{2} =\left\{\left|B^{2} -\frac{4}{27} \right|\le \varepsilon \right\}.\]

First, consider the estimate for the integral when $B$ lies in the set $A_{1} .$ In this case equation  $y^{3} -y+B=0$ has three different roots $y_{1} (B),y_{2} (B),y_{3} (B),$ since  $|y_{j} (B)-y_{k} (B)|\ge \Delta (\varepsilon )>0$ when $j\ne k.$ Since $\delta <1,$ then the following integral convergences
\begin{equation} \label{GrindEQ__3_10_}
J_{10} :=\int _{-2}^{2} \frac{dy}{|(y-y_{1} (B))(y-y_{2} (B))(y-y_{3} (B))|^{\delta } }.
\end{equation}
Since $B$ belongs to the compact set $A_{1} $, the integral \eqref{GrindEQ__3_10_} is uniformly bounded for $\delta <1.$ Therefore, we have
\[\left|J\right|=\frac{1}{\left|p\right|^{\frac{3\delta -1}{2} } } \int _{-\left|p\right|^{\frac{-1}{2} } }^{\left|p\right|^{\frac{1}{2} } } \frac{dy}{\left|y^{3} -y+B\right|^{\delta } } \le \frac{c_{\delta } }{\left|p\right|^{\frac{3\delta -1}{2} } } ,\]
as $B\in A_{1} .$ Now we consider the integral in the case when $B\in A_{2}.$ For the sake of definiteness, we can assume that $\left|B-\frac{2}{3\sqrt{3} } \right|<\varepsilon,$ where $0<\varepsilon <\frac{1}{6\sqrt{3}}$ a fixed positive number. Note that $y_{1,2} =\pm \frac{1}{\sqrt{3} } $ are critical points of the function $F(y)=y^{3}-y+B.$   Thus, the number $B=B_{0} =\frac{2}{3\sqrt{3} } $ for the function $F(y,B_{0} )$ is a simple critical value. Let us study the behavior of the integral as $\left|B-B_{0} \right|<\Delta ,$ where $\Delta$ is sufficiently small fixed positive number. Since $F\left(\frac{1}{\sqrt{3} } \right)=0,$ then $F\left(y,B_{0} \right)=\left(y-\frac{1}{\sqrt{3} } \right)^{2} \left(y+\frac{2}{\sqrt{3} } \right)$. Hence,
\[F\left(y+\frac{1}{\sqrt{3} } ,B\right)=y^{3} +\sqrt{3} y^{2} -\frac{2}{3\sqrt{3} } +B=y^{2} \left(y+\sqrt{3} \right)+H,\]
where $H=B-\frac{2}{3\sqrt{3} } $.

We consider the following integral
\begin{equation} \label{GrindEQ__3_11_}
\int _{-\sigma }^{\sigma }\frac{dy}{\left|y^{2} \left(y+\sqrt{3} \right)+H\right|^{\delta } }  ,
\end{equation}
where $\sigma >0$ sufficiently small fixed positive number. If $\frac{1}{3} <\delta <\frac{1}{2} ,$ then according to the \cite{Ricci and Stein}, the integral \eqref{GrindEQ__3_11_} is uniformly bounded with respect to $H$.  This completes the proof of the first part of Theorem \ref{th203}.

Next, suppose that $\frac{1}{2} <\delta <1.$ We use change of variables as $z=y\sqrt{y+\sqrt{3} } $ and  and denoting the inverse function by $y=y\left(z\right),$ we have:
\[\int _{-\sigma }^{\sigma }\frac{dy}{\left|y^{2} \left(y+\sqrt{3} \right)+H\right|^{\delta } }  =\int _{-\sigma _{1} }^{\sigma _{2} } \frac{y^{'} (z)dz}{\left|z^{2} +H\right|^{\delta } } =\]
\[=\frac{1}{\sqrt{3} } \int _{-\sigma _{1} }^{\sigma _{2} } \frac{dz}{\left|z^{2} +H\right|^{p} } +c\int _{-\sigma _{1} }^{\sigma _{2} } \frac{z\varphi (z)dz}{\left|z^{2} +H\right|^{\delta } } =J_{1}^{'} +J_{2}^{'} \, \, ,\]
where $\varphi$ is any smooth function, $\sigma _{1} $ and $\sigma _{2} $ determined from the conditions $y(\sigma _{1} )=-\sigma ,$ $y(\sigma _{2} )=\sigma .$ In the integrals $J_{1}^{'} $ and $J_{2}^{'} $ we make a linear change $z=\left|H\right|^{1/2} t$. Then
\begin{equation} \label{GrindEQ__3_12_}
J_{1}^{'} =\frac{1}{\sqrt{3} \left|H\right|^{\delta -1/2} } \int _{-\sigma _{1} \left|H\right|^{-1/2} }^{\sigma _{2} \left|H\right|^{-1/2} } \frac{dt}{\left|t^{2} \pm 1\right|^{\delta } } .
\end{equation}
It is easy to show that the integral \eqref{GrindEQ__3_12_} is uniformly bounded with respect to $H$, for $\frac{1}{2} <\delta <1.$ From here we get:
\[|J_{1}^{'} |\le \frac{c}{H^{\delta -1/2} } .\]
The integral $J_{2}^{'}$ is estimated as follows:
\[J_{2}^{'} =c\int _{-\sigma _{1} }^{\sigma _{2} } \frac{z\varphi (z)dz}{\left|z^{2} +H\right|^{\delta } } =c\int _{-\sigma _{1} /H^{1/2} }^{\sigma _{2} /H^{1/2} } \frac{Ht\varphi (\left|H\right|^{1/2} t)dt}{\left|Ht^{2} +H\right|^{\delta } } \le \]
\[\le \frac{c}{H^{\delta -1} } \int _{-\sigma _{1} /H^{1/2} }^{\sigma _{2} /H^{1/2} } \frac{t\varphi (\left|H\right|^{1/2} t)dt}{\left|t^{2} \pm 1\right|^{\delta } } \le \frac{c_{1} }{H^{\delta -1} } \int _{2}^{\left|H\right|^{-1/2} } t^{1-2\delta } dt+\frac{c_{2} }{H^{\delta -1} } =\]
\[=\frac{c_{1} H^{1-\delta } }{2-2\delta } \left. t^{2-2\delta } \right|_{2}^{\left|H\right|^{-1/2} } +\frac{c_{2} }{H^{\delta -1} } =\frac{c_{1} H^{1-\delta } }{2-2\delta } \left[H^{\delta -1} -2^{2-2\delta } \right]+\frac{c_{2} }{H^{\delta -1} } \le C,\]
the validity of the last inequality follows from the condition $\delta <1.$ Thus, the integral $J_{2}^{'} $ is uniformly bounded. Summing up the obtained estimates, we have:
\[\int _{-\sigma }^{\sigma } \frac{dy}{\left|y^{2} \left(y+\sqrt{3} \right)+H\right|^{\delta } } \le \frac{c}{\left|H\right|^{\delta -1/2} } .\]
Thus, under the conditions \eqref{GrindEQ__3_5_} and $\frac{1}{2} <\delta <1$, we get the estimate:
\[\left|J(p,q)\right|=\int _{-1}^{1} \frac{dx}{\left|x^{3} +px+q\right|^{\delta } } \le \frac{1}{|p|^{\frac{3\delta -1}{2} } } \frac{c}{\left|B-\frac{2}{3\sqrt{3} } \right|^{\delta -1/2} } =\frac{1}{|p|^{\frac{3\delta -1}{2} } } \times {\kern 1pt} {\kern 1pt} {\kern 1pt} {\kern 1pt} {\kern 1pt} {\kern 1pt} {\kern 1pt} {\kern 1pt} {\kern 1pt} {\kern 1pt} {\kern 1pt} {\kern 1pt} {\kern 1pt} {\kern 1pt} {\kern 1pt} {\kern 1pt} \]
\[\times \frac{c}{\left|\frac{q}{\left(-p\right)^{3/2} } -\frac{2}{3\sqrt{3} } \right|^{\delta -1/2} } \le \frac{c}{\left|p\right|^{1/4} } \frac{\left|q+\frac{2}{3\sqrt{3} } \left(-p\right)^{3/2} \right|^{\delta -1/2} }{\left|D\right|^{\delta -\frac{1}{2} } } \le \frac{c}{D^{\delta -\frac{1}{2} } \left(\frac{\left|p\right|^{3} }{27} +\frac{q^{2} }{4} \right)^{\frac{2-3\delta }{6} } } .\]

Now suppose that the condition \eqref{GrindEQ__3_6_} holds true. In this case, we use the change of variables $x=\left|q\right|^{1/3} z$ for the integral \eqref{integ2}. Then
\[\left|J(p,q)\right|=\frac{1}{\left|q\right|^{\delta -\frac{1}{3} } } \int _{-\left|q\right|^{\frac{-1}{3} } }^{\left|q\right|^{\frac{1}{3} } } \frac{dy}{\left|z^{3} +Az+sgn(q)\right|^{\delta } } ,\]
where $\left|A\right|=\left|\frac{p}{q^{2/3} } \right|\le \frac{3}{\sqrt[{3}]{4} } $.

If $A\in \left[-\frac{3}{\sqrt[{3}]{4} } +\varepsilon ,\frac{3}{\sqrt[{3}]{4} } -\varepsilon \right]$ (where $\varepsilon $ sufficiently small positive fixed number), then again the integral is uniformly bounded. Case $\left|A^{2} -\frac{9}{\sqrt[{3}]{16} } \right|<\varepsilon $ is considered similarly to the case $\left|B^{2} -\frac{4}{27} \right|<\varepsilon .$

The case when $\delta=\frac{1}{3}$ is trivially holds and we can obtain an analogical estimate for the case when $\delta=\frac{1}{2}$ considering separately the cases $|x^{3}+px+q|>1$ and $|x^{3}+px+q|\leq1$. \textbf{Theorem \ref{th203} is proved.}

\section{Invariant estimates for homogeneous polynomial phase}

In this section, we consider generalized oscillatory integrals with phase function which is a homogeneous polynomial of degree three in two variables
\begin{equation}\label{equation22}
P_{3}(a,x)=a_{0}x_{1}^{3}+3a_{1}x_{1}^{2}x_{2}+3a_{2}x_{1}x_{2}^{2}+a_{3}x_{2}^{3}.
\end{equation}
We consider the following integral
\begin{equation}\label{int161}
I_{\alpha,\beta}=\int _{|x|\leq1}E_{\alpha,\beta}(iP_{3}(a,x))\psi(x)dx_{1}dx_{2},
\end{equation}
where $|x|\leq1$ is the unite circle.
We consider behavior of the integral \eqref{int161} in the case when the coefficients of the polynomial tend to infinity. Let us obtain estimates for integral \eqref{int161} in terms of the invariants of the group of motions of the Euclidean plane. Note that, the discriminant of a polynomial denoted by $D$ is defined by the formula:
$$D=3a_{1}^{2}a_{2}^{2}+6a_{0}a_{1}a_{2}a_{3}-4a_{0}a_{2}^{3}-4a_{1}^{3}a_{3}-a_{0}^{2}a_{3}^{2}$$
it is the invariant of the group $SL(2,\mathbb{C})$.

\begin{theorem}
For the integral \eqref{int16} with phase \eqref{equation22} following inequality
$$|I_{\alpha,\beta}|\leq\frac{c\|\psi\|_{\mathbb{L^{\infty}}}}{|D|^{\frac{1}{6}}},$$
holds, where $c$ is a constant and $D$ is discriminant of the polynomial $P_{3}$.
\end{theorem}
\textbf{Proof.} For the sake of defined we suppose that $|a_{0} |=max\{ \left|a_{i} \right|,i=\overline{0,3}\} ,$ otherwise, by rotating the coordinate axes, we can reduce the general case to the case under consideration.
Since $D$ is the invariant of the group $SL(2,\mathbb{C}),$ the integral and estimate do not depend on the choice of such a change of variables. Moreover, the norm of the amplitude is uniformly bounded because the rotation group is compact. Let's represent the polynomial in the form
\[P_{3} =a_{0} \left(x^{3} +\frac{3a_{1} }{a_{0} } x^{2} y+\frac{3a_{2} }{a_{0} } xy^{2} +\frac{a_{3} }{a_{0} } y^{3} \right).\]

Let we make a change of variables as follows $x_{1} =x+\frac{a_{1} }{a_{0} } y,$ $y_{1} =y.$ Since $\left|\frac{a_{1 } }{a_{0} } \right|\le 1,$ then the Jacobian and the transformation norm are uniformly bounded, which is important in what follows. As a result, we get
\[P_{3} =a_{0} \left(x_{1}^{3} +px_{1} y_{1}^{2} +qy_{1}^{3} \right)=a_{0} \Phi ,\]
where $p=\frac{3a_{0} a_{2} -3a_{1}^{2} }{a_{0}^{2} } $ и $q=\frac{a_{0}^{2} a_{3} +2a_{1}^{3} -3a_{0} a_{1} a_{2} }{a_{0}^{3} } .$ Note that $|p|\le 6,$ $|q|\le 6.$

Hence the integral has the form
\begin{equation} \label{GrindEQ__4_1_}
J_{\alpha,\beta}:=\int  E_{\alpha,\beta}(ia_{0} \Phi) \psi (x_{1} ,y_{1} )dx_{1} dy_{1} .
\end{equation}
Applying the polar coordinate system $x_{1} =r\cos \theta ,{\kern 1pt} {\kern 1pt} {\kern 1pt} {\kern 1pt} y_{1} =r\sin \theta ,$ we get
\begin{equation} \label{GrindEQ__4_2_}
J_{\alpha,\beta}\leq\int _{0}^{2\pi } \int _{0}^{\infty } E_{\alpha,\beta}(ia_{0} r^{3} \phi (\cos \theta ,\sin \theta )) r\psi (r\cos \theta ,r\sin \theta )drd\theta .
\end{equation}
For the inner integral \eqref{GrindEQ__4_2_}, i.e. for
\[J_{in} :=\int _{0}^{\infty } E_{\alpha,\beta}(ia_{0} r^{3} \phi (\cos \theta ,\sin \theta )) r\psi (r\cos \theta ,r\sin \theta )dr\]
using Proposition \ref{Prop1}, we get
$$|J_{in}|\leq\int_{0}^{\infty}\frac{r|\psi (r\cos \theta ,r\sin \theta )|dr}{1+|a_{0}r^{3}\phi(\cos\theta,\sin\theta)|}.$$
We make change the variable as $\rho=r^{3}|a_{0}\phi(\cos\theta,\sin\theta)|$. So, we have
\[|J_{in} |\le \frac{c\left\| \psi \right\| _{\mathbb{L^{\infty}} } }{|a_{0} \phi (\cos \theta ,\sin \theta )|^{\frac{2}{3} } }\int_{0}^{\infty}\frac{d\rho}{\rho^{\frac{1}{3}}(1+\rho)} .\]
As the last integral convergence, we obtain
\[|J_{in} |\le \frac{c\left\| \psi \right\|_{\mathbb{L^{\infty}} } }{|a_{0} \phi (\cos \theta ,\sin \theta )|^{\frac{2}{3} } }.\]

Thus, the integral $J$ has the estimate
\begin{equation} \label{GrindEQ__4_3_}
|J_{\alpha,\beta}|\le \frac{c\left\| \psi \right\| _{\mathbb{L^{\infty}} } }{|a_{0} |^{\frac{2}{3} } } \int _{0}^{2\pi } \frac{d\theta }{|\phi (\cos \theta ,\sin \theta )|^{\frac{2}{3} } } ,
\end{equation}
where $\phi =\mathop{\cos }\nolimits^{3} \theta +p\cos \theta \mathop{\sin }\nolimits^{2} \theta +q\mathop{\sin }\nolimits^{3} \theta .$

We introduce the following integral
\begin{equation} \label{GrindEQ__4_4_}
J_{2} :=\int _{0}^{2\pi } \frac{d\theta }{|\phi (\cos \theta ,\sin \theta )|^{\frac{2}{3} } } ,
\end{equation}
where $\phi =\mathop{\cos }\nolimits^{3} \theta +p\cos \theta \mathop{\sin }\nolimits^{2} \theta +q\mathop{\sin }\nolimits^{3} \theta .$

\begin{lemma}\label{lemma2}
\textit{For the integral \eqref{GrindEQ__4_4_} following estimate holds true}
\[|J_{2} |=\left|\int _{0}^{2\pi } \frac{d\theta }{|\mathop{\cos }\nolimits^{3} \theta +p\cos \theta \mathop{\sin }\nolimits^{2} \theta +q\mathop{\sin }\nolimits^{3} \theta |^{\frac{2}{3} } } \right|\le \frac{c}{|D(\phi )|^{\frac{1}{6} } } ,\]
\textit{where $D(\phi )=\frac{p^{3} }{27} +\frac{q^{2} }{4} .$}\textbf{\textit{}}
\end{lemma}
Lemma \ref{lemma2} was proved in \cite{Saf Bashkir}. For the convince of readers we give a proof of Lemma \ref{lemma2}.

\textbf{Proof of Lemma \ref{lemma2}.} Note, that
\[J_{2} =2\int _{0}^{\pi } \frac{d\theta }{|\mathop{\cos }\nolimits^{3} \theta +p\cos \theta \mathop{\sin }\nolimits^{2} \theta +q\mathop{\sin }\nolimits^{3} \theta |^{\frac{2}{3} } } =2(J_{2_{0} } +J_{2_{1} } ),\]
where $J_{2_{0} } =\int _{0}^{\frac{\pi }{2} } \frac{d\theta }{|\mathop{\cos }\nolimits^{3} \theta +p\cos \theta \mathop{\sin }\nolimits^{2} \theta +q\mathop{\sin }\nolimits^{3} \theta |^{\frac{2}{3} } }, $  $J_{2_{1} } =\int _{\frac{\pi }{2} }^{\pi } \frac{d\theta }{|\mathop{\cos }\nolimits^{3} \theta +p\cos \theta \mathop{\sin }\nolimits^{2} \theta +q\mathop{\sin }\nolimits^{3} \theta |^{\frac{2}{3} } }.$

First, we consider
\[J_{2_{0} } =\int _{0}^{\frac{\pi }{2} } \frac{d\theta }{|\mathop{\cos }\nolimits^{3} \theta +p\cos \theta \mathop{\sin }\nolimits^{2} \theta +q\mathop{\sin }\nolimits^{3} \theta |^{\frac{2}{3} } } =\]
\[=\int _{0}^{\frac{\pi }{4} } \frac{d\theta }{|\mathop{\cos }\nolimits^{3} \theta +p\cos \theta \mathop{\sin }\nolimits^{2} \theta +q\mathop{\sin }\nolimits^{3} \theta |^{\frac{2}{3} } } +\int _{\frac{\pi }{4} }^{\frac{\pi }{2} } \frac{d\theta }{|\mathop{\cos }\nolimits^{3} \theta +p\cos \theta \mathop{\sin }\nolimits^{2} \theta +q\mathop{\sin }\nolimits^{3} \theta |^{\frac{2}{3} } } .\]
Then
\[J_{2_{0} }^{'} =\int _{0}^{\frac{\pi }{4} } \frac{d\theta }{|\mathop{\cos }\nolimits^{3} \theta +p\cos \theta \mathop{\sin }\nolimits^{2} \theta +q\mathop{\sin }\nolimits^{3} \theta |^{\frac{2}{3} } } \]
making the change of variables $tg\theta =t,$ we get
\begin{equation} \label{GrindEQ__4_5_}
J_{2_{0} }^{'} =\int _{0}^{1} \frac{dt}{|qt^{3} +pt^{2} +1|^{\frac{2}{3} } } .
\end{equation}

Let $|t|\le \frac{1}{4} $ as $|p|\le 6,{\kern 1pt} {\kern 1pt} {\kern 1pt} {\kern 1pt} \; |q|\le 6,$ then $|qt^{3} |\le \frac{6}{64} ,$ $|pt^{2} |\le \frac{6}{16} .$ Hence,

\noindent $\left|t^{3} +pt^{2} \right|\le \frac{30}{64} $ и $\left|1+qt^{3} +pt^{2} \right|\ge 1-\frac{30}{64} =\frac{34}{64} .$

Then
\[\left|J_{2_{0} }^{'} \right|=\int _{0}^{1} \frac{dt}{|qt^{3} +pt^{2} +1|^{\frac{2}{3} } } =\int _{0}^{\frac{1}{4} } \frac{dt}{|qt^{3} +pt^{2} +1|^{\frac{2}{3} } } +\]
\[+\int _{\frac{1}{4} }^{1} \frac{dt}{|qt^{3} +pt^{2} +1|^{\frac{2}{3} } } \le 1+\int _{\frac{1}{4} }^{1} \frac{dt}{|qt^{3} +pt^{2} +1|^{\frac{2}{3} } } .\]
By changing the variables $x=\frac{1}{t} $ the last integral is reduced to the form:
\begin{equation} \label{GrindEQ__4_6_}
\int _{\frac{1}{4} }^{1} \frac{dt}{|qt^{3} +pt^{2} +1|^{\frac{2}{3} } } =\int _{1}^{4} \frac{dx}{|x^{3} +px+q|^{\frac{2}{3} } } .
\end{equation}
It is obvious that the equality
\[\int _{\frac{\pi }{4} }^{\frac{\pi }{2} } \frac{d\theta }{|\mathop{\cos }\nolimits^{3} \theta +p\cos \theta \mathop{\sin }\nolimits^{2} \theta +q\mathop{\sin }\nolimits^{3} \theta |^{\frac{2}{3} } } =\int _{0}^{1} \frac{dx}{|x^{3} +px+q|^{\frac{2}{3} } } .\]
Thus, the problem reduces to estimating an integral of the form
\begin{equation} \label{GrindEQ__4_7_}
\int _{N_{1} }^{N_{2} } \frac{dx}{|x^{3} +px+q|^{\frac{2}{3} } } ,
\end{equation}
where $N_{1} ,{\kern 1pt} {\kern 1pt} N_{2} $ are fixed numbers. Finally, the desired estimate for the integral \eqref{GrindEQ__4_7_} follows easily from Theorem \ref{th203}. The estimate for the integral $J_{2_{1}}^{'} $ is similarly performed. Which completes the proof of Lemma \ref{lemma2}.

Applying Lemma \ref{lemma2} for the integral \eqref{GrindEQ__4_3_}, we obtain the required estimate. \textbf{Theorem 5.1 is proved.}

\textbf{Declaration of competing interest.}

This work does not have any conflicts of interest.

\subsection*{Data availability} My manuscript has no associated data.


\end{document}